\documentclass[a4paper, 11pt]{article}
\usepackage{amsmath}
\usepackage{amsfonts}
\usepackage{textcomp}
\usepackage{mathcomp}

\usepackage{color,ifpdf,latexsym,url}
\ifpdf
\usepackage[pdftex]{graphicx}
\fi
\usepackage  [breaklinks,bookmarks,bookmarksnumbered,bookmarksopen,bookmarksopenlevel=2]
  {hyperref}
\hypersetup{
    colorlinks=true,
    linkcolor=blue,
    urlcolor=cyan,
    citecolor=blue
    }

{\makeatletter \hypersetup{pdftitle={\@title}}}
\usepackage{siunitx}

\def\authornames{Hussein NASSAR and Andrew WEBER} 

\usepackage{iass}

\usepackage[backend=biber,sorting=none,style=ieee,citestyle=numeric]{biblatex}
\newcommand{\tblcaption}[1]{\def\@captype{table}\caption{#1}}

\usepackage{bm}

\newcommand{\ten}{\mathbf}
\newcommand{\gt}{\bm}
\newcommand{\scalar}[1]{\langle#1\rangle}

\newcommand{\str}{\varepsilon}

\renewcommand{\[}{\begin{equation*}}
\renewcommand{\]}{\end{equation*}}

\usepackage{amsthm}
\newtheorem{theorem}{\bf Theorem}
\newtheorem{example}{\bf Example}

\newtheorem{definition}{\bf Definition}

\begin{document}

\thispagestyle{firstpage}
\pagestyle{fancy}

\vspace*{0mm}
\begin{center}%

{\LARGE \bfseries How periodic surfaces bend without stretching\\}%
\vspace{1.0em}
{\normalsize Hussein NASSAR$^\textrm{*}$, Andrew WEBER\\}
\vspace{1.0em}
{\small Department of Mechanical and Aerospace Engineering, University of Missouri,\\ Columbia, 65211 MO, USA\\ $^\textrm{*}$nassarh@missouri.edu\\}
\vspace{1em}

\end{center}%
\begin{abstract} Many compliant shell mechanisms are periodically corrugated or creased. Being thin, their preferred deformation modes are inextensional, i.e., isometric. Here, we report on a recent characterization of the isometric deformations of periodic surfaces. In a way reminiscent of Gauss theorem, the result builds a constraint that relates the ways in which the periodic surface stretches, effectively but isometrically, to the ways in which it bends and twists. Several examples and use cases are presented.
\end{abstract}
  
\vskip 1.0em

\section{Main result}
\begin{definition}
    A periodic surface $\ten x: (x,y)\mapsto (x,y,f(x,y))$ is the graph of some function $f$ that is periodic in $x$ and $y$. 
\end{definition}
We are thinking of such a surface as the midsurface of a thin shell. Function $f$ could be smooth and describe a smoothly ``corrugated'' shell. It could be piecewise smooth and describe a shell with curved creases (Figure~\ref{fig:main}). It could also be piecewise linear and describe an origami or origami-like tessellation. The main premise of the present work is that the compliant deformation modes of the thin shell are in correspondence with the \emph{isometric deflections} of the midsurface.

\begin{definition}
    An \emph{isometric deflection} of a surface $\ten x: (x,y)\mapsto (x,y,f(x,y))$ is a displacement field $\dot{\ten x}: (x,y)\mapsto (u(x,y),v(x,y),w(x,y))$ such that the infinitesimal strains
    \[
        \str_{11} \equiv \scalar{\partial_x\ten x,\partial_x \dot{\ten x}}, \quad
        \str_{22} \equiv \scalar{\partial_y\ten x,\partial_y \dot{\ten x}}, \quad
        \str_{12} \equiv \frac{\scalar{\partial_x\ten x,\partial_y \dot{\ten x}}+\scalar{\partial_y\ten x,\partial_x \dot{\ten x}}}{2}
    \]
    are zero. Therein, the brackets $\scalar{\cdot\,,\cdot}$ denote the usual dot product.
\end{definition}

\begin{figure}
    \centering
    \includegraphics[width=\linewidth]{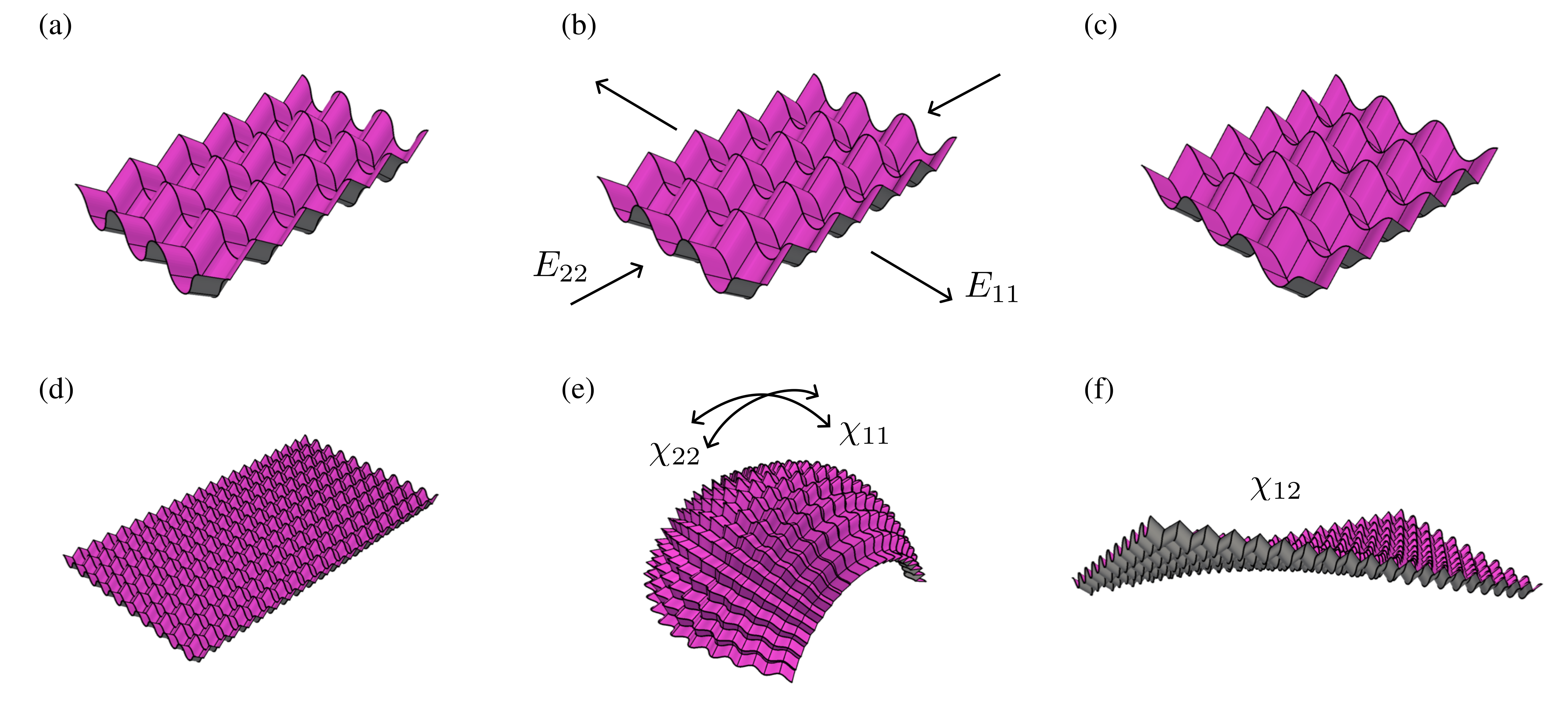}
    \caption{A periodic surface of the form $(x,y,f(x)+g(y))$ where $f$ is a sine function and $g$ is a triangular pattern: (a-c) an effective membrane mode where the surface stretches without bending by bending without stretching; (d) the same surface with more unit cells to help appreciate the curvatures seen in (e) and (f). Theorem~\ref{thm} states that (b) implies (e) and that the possibility of (f) implies the impossibility of an effective shear membrane mode. Surfaces are computed using a triangulation; source code available at \url{https://github.com/nassarh/curvedCreaseOrigami}.}
    \label{fig:main}
\end{figure}

The reason behind that correspondence between \emph{isometric} and \emph{compliant} is that isometric deflections only engage the flexure rigidity of the shell which is of order~$t^3$ with~$t$ being the thickness~\cite{Landau1986}. By contrast, non-isometric deflections are much stiffer since they also engage the membrane rigidity of order~$t$. That being said, it would be excessive to claim that any isometric deflection describes a compliant \emph{mode}, for within the word ``mode'' is typically embedded a notion of ``uniformity'' whereby the mode acts on the shell as a whole and not on some restricted parts of it (say a corner or some small flat piece). Here, we focus on two specific modes: effective membrane modes and effective flexure modes.

\begin{definition}
    Let $\ten x: (x,y)\mapsto (x,y,f(x,y))$ be a periodic surface. An \emph{effective membrane mode} is an isometric deflection $\dot{\ten x}: (x,y)\mapsto (u(x,y),v(x,y),w(x,y))$ such that ($i$) $w$ is periodic and $(ii)$ the effective strains
    \[
        E_{11} \equiv \scalar{\ten p_1,\dot{\ten p}_1},\quad
        E_{22} \equiv \scalar{\ten p_2,\dot{\ten p}_2},\quad
        E_{12} \equiv \frac{\scalar{\ten p_1,\dot{\ten p}_2}+\scalar{\ten p_2,\dot{\ten p}_1}}{2}
    \]
    are \emph{not} all zero. Therein, vectors $(\ten p_1,\ten p_2)$ outline a unit cell and vectors $(\dot{\ten p}_1,\dot{\ten p}_2)$ are their deflections, namely
    \[\begin{split}
        \ten p_1 &= \ten x(L_1,0) - \ten x(0,0),\quad
        \ten p_2 = \ten x(0,L_2) - \ten x(0,0),\\
        \dot{\ten p}_1 &= \dot{\ten x}(L_1,0) - \dot{\ten x}(0,0),\quad
        \dot{\ten p}_2 = \dot{\ten x}(0,L_2) - \dot{\ten x}(0,0),
    \end{split}
    \]
    where $[0,L_1]\times[0,L_2]$ is a unit cell.
\end{definition}

Thus, for an isometric deflection to be an effective membrane mode, it must $(i)$ produce the same deflections within all unit cells and $(ii)$ must leave a visible trace on the outline of the unit cell; see Figure~\ref{fig:main}a-c.

\begin{definition}
    Let $\ten x: (x,y)\mapsto (x,y,f(x,y))$ be a periodic surface. An \emph{effective flexure mode} is an isometric deflection $\dot{\ten x}: (x,y)\mapsto (u(x,y),v(x,y),w(x,y))$ such that $w$ is periodic up to a non-zero quadratic displacement $\chi_{11} x^2/2 + \chi_{22} y^2/2 + \chi_{12}xy$ of effective curvatures $\gt \chi$.
\end{definition}

Here too, all unit cells deform in the same fashion except for a quadratic deflection that describes an overall bent state; see Figure~\ref{fig:main}d-f. The main result we want to report here is that effective strains and effective curvatures constrain each other. That is: the ways in which the periodic shell effectively stretches, contracts and shears in the plane constrain the ways in which it could effectively bend and twist out of plane, and vice versa; see Figure~\ref{fig:main}.

\begin{theorem}\label{thm}
    Let $\ten x: (x,y)\mapsto (x,y,f(x,y))$ be a periodic surface that admits an effective strain $\ten E$ and effective curvatures $\gt\chi$. Then,
    \[
        E_{11}\chi_{22} -2 E_{12}\chi_{12} + E_{22}\chi_{11} = 0. 
    \]
\end{theorem}
\begin{proof}
    Admitted. See~\cite{NASSAR2024105553}.
\end{proof}

\section{Previous work}
Several contributions have proven Theorem~\ref{thm} in particular cases of increasing complexity all within the context of origami and origami-like tessellations. Therein, the theorem is often stated as an equality (up to a sign) between an ``in-plane Poisson's coefficients''
\[
    \nu_\text{in} \equiv -\frac{E_{22}}{E_{11}},
\]
and an ``out-of-plane Poisson's coefficient''
\[
    \nu_\text{out} \equiv -\frac{\chi_{22}}{\chi_{11}},
\]
namely that
\[
    \frac{E_{22}}{E_{11}} = -\frac{\chi_{22}}{\chi_{11}}.
\]
In words: \emph{the relative ratio of effective extensions is equal and opposite to the ratio of effective normal curvatures}, the identity being valid in cases where $E_{12}=0$. The first proofs of this identity were obtained for the Miura ori~\cite{Wei2013, Schenk2013}, then for the ``eggbox pattern''~\cite{Nassar2017a,Nassar2017e} and were soon extended to the ``morph'' pattern~\cite{Pratapa2019} and more recently to any origami tessellation with four parallelogram panels per unit cell~\cite{Nassar2022,McInerney2022}. In all of these cases, the proof involves a computation that heavily relies on the discreteness and simplicity of the unit cell and is hard to generalize, through an equally tractable computation, to other surfaces with more panels per unit cell, with curved creases or with smoothly curved, non-polyhedral, features. In comparison, Theorem~\ref{thm} crushes at once all cases covered by the hypothesis of piecewise smoothness using high level concepts from differential geometry and avoids any ``brute force'' computations relying, say, on spherical trigonometry.

In order to give a taste of the kinds of results that Theorem~\ref{thm} relies on, let for simplicity $f$ be a smooth function, and recall that the Gauss curvature of the graph of $f$ is
\[
    K_f = \frac{\partial_{xx}f\partial_{yy}f-(\partial_{xy}f)^2}{\sqrt{1+(\partial_x f)^2+(\partial_y f)^2}}.
\]
By Gauss theorem, $K_f$ is an isometric invariant meaning that for an isometric deflection $(u,v,w)$, one has $K_f=K_{f+w}$ which implies
\[\label{eq:deltaK}
    \partial_{yy}f\partial_{xx}w-2\partial_{xy}f\partial_{xy}w+\partial_{xx}f\partial_{yy}w = 0.
\]
Theorem~\ref{thm} then leverages a self-adjointness property of this equation to justify an averaging step that transforms it into the desired identity. Therein, two ``symmetries'' are crucial: $(i)$ if $w$ is an isometric deflection of $f$ then $f$ is an isometric deflection of $w$. Indeed, the perturbation $K_f-K_{f+w}$ turns out to be symmetric in $f$ and $w$. And, $(ii)$ the Hessian of $f$ is a symmetric matrix. The full proof can be found in~\cite{NASSAR2024105553}.

\section{Use cases}
Morally, Theorem~1 establishes some sort of a ``principle of conservation'' between rigidity and flexibility. Here is how.
\begin{example}
    The plane $(x,y)\mapsto(x,y,0)$ admits the effective flexure mode $(x,y)\mapsto(0,0,x^2/2)$ whose effective curvatures are
    \[
        [\gt\chi] = \begin{bmatrix}1 & 0 \\ 0 & 0\end{bmatrix}.
    \]
    Thus, by Theorem~\ref{thm}, any effective strain $\ten E$ satisfies
    \[
        E_{22} = 0.
    \]
    Similarly, $E_{11}$ and $E_{12}$ are null meaning that the plane admits no effective membrane modes. No surprises here. It is just reassuring that the implications of the theorem are consistent with expectations in such simple cases.
\end{example}

\begin{example}
    The simply corrugated plane $(x,y)\mapsto (x,y,\cos(x))$ can still effectively bend in the $x$-direction meaning that it cannot effectively stretch in the $y$-direction: $E_{22}=0$. By contrast, the corrugation in the $x$-direction enables an effective membrane mode that stretches in the $x$-direction: $E_{11}\neq 0$. Then, by the theorem, there can be no effective bending in the $y$-direction: $\chi_{22}=0$. Again, this is consistent with expectations: corrugating in one direction stifles bending in the orthogonal direction. But the interpretation is quite novel and remarkable: the corrugated plane gains an effective membrane mode, and therefore loses and effective flexure mode.
\end{example}

\begin{example}
    Any doubly corrugated surface $(x,y)\mapsto (x,y,f(x)+g(y))$ admits the effective flexure mode $(x,y)\mapsto (0,0,xy)$ with effective curvature $\chi_{12}=1$. Indeed, $w=xy$ solves equation~\ref{eq:deltaK}, namely
    \[
        f''\partial_{yy}w + g''\partial_{xx}w=0
    \]
    Thus, by the theorem, effective membrane modes are all shearless: $E_{12}=0$. In other words, all doubly corrugated surfaces resist shear.
\end{example}

\begin{example}
    Similarly: the Miura ori twists easily indicating that it admits an effective flexure mode with $\chi_{12}\neq 0$. Thus, by the theorem, the Miura ori resists shear: $E_{12}=0$. The observation should be easily verifiable by anyone with access to a Miura ori: it twists easily but shears hardly, if at all.
\end{example}

\begin{example}
    Theorem~\ref{thm} also generalizes the identity of the Poisson's coefficients to cases where $E_{12}\neq 0$. Suppose a surface admits an effective membrane mode of effective strain $\ten E$ and an effective flexure mode of curvatures $\ten \chi$. Write the components of $E'_{\alpha\beta}$ and $\chi'_{\alpha\beta}$ in a basis aligned with the principal directions of strain. Then $E'_{12}=0$ and
    \[
        \frac{E'_{22}}{E'_{11}} = -\frac{\chi'_{22}}{\chi'_{11}}.
    \]
    In particular all ``auxetic'' periodic shells bend anticlastically and all ``anauxetic'' periodic shells bend synclastically (in the directions of principal effective strains).
\end{example}

\section{Extensions and limitations}
Here, we adopted an \emph{infinitesimal} notion of isometric deflections. There exists a variant of Theorem~\ref{thm} that applies in a geometrically non-linear setting, i.e., for \emph{finite} isometries~\cite{NASSAR2024105553}. In that case, Theorem~\ref{thm} provides an equation that describes the shape of a periodic shell resulting from a finite effective isometric deformation; see examples investigated in~\cite{Nassar2017a,Nassar2018b}.

It should be highlighted that Theorem~\ref{thm} only informs on the interactions between membrane and flexure modes \emph{should they exist}. But it does not provide existence conditions for either. Furthermore, Theorem~\ref{thm} only applies to \emph{closed} surfaces, i.e., surfaces without cutouts. Indeed, cutouts introduce boundary conditions that the present geometric approach appears to be ill-equipped to handle.

Finally, the theorem implicitly presumes the existence of two scales: a small scale that corresponds to the scale of the unit cell, and a large scale that corresponds to the scale of the effective strains and curvatures. However, recent investigations revealed that there is a third intermediate scale that emerges in cases where a periodic surface has effective flexure modes but no effective membrane modes (e.g., Ron Resch pattern, Yoshimura pattern, etc.); see \cite{Reddy2023, Imada2023}.

\section*{Acknowledgments}
Work supported by the (U.S.) NSF under CAREER award No.~CMMI-2045881.

\printbibliography
\end{document}